\documentclass[12pt]{amsart}

\usepackage{nicematrix}
\usepackage{calligra,mathrsfs}
\usepackage[all]{xy}
\usepackage{float, comment}
\usepackage{mathtools}
\usepackage{amsmath}
\usepackage{amsthm}
\usepackage{amssymb}
\usepackage{amsbsy}
\usepackage{amstext}
\usepackage{amsopn}
\usepackage{mathrsfs} 
\usepackage{enumerate}
\usepackage{xcolor}
\usepackage{graphicx} 
\usepackage{microtype} 
\usepackage[margin=1in,marginparwidth=0.8in, marginparsep=0.1in]{geometry}
\usepackage[pagebackref, bookmarks=true, bookmarksopen=true, bookmarksdepth=3,bookmarksopenlevel=2, colorlinks=true, linkcolor=blue, citecolor=blue, filecolor=blue, menucolor=blue, urlcolor=blue]{hyperref}
\usepackage{tikz}
\usepackage{bbm}
\usepackage[all]{xy}
\usetikzlibrary{arrows,calc,positioning,decorations.pathreplacing} 

\usepackage[shortlabels]{enumitem}

\numberwithin{equation}{section}

\usepackage{todonotes}

\usepackage{xcolor}

\usepackage{stmaryrd}

\usepackage{tikz-cd}

\usepackage{enumitem}

\usepackage{quiver}

\numberwithin{equation}{section}
\newtheorem{thm}{Theorem}[section]

\newtheorem{definition}[thm]{Definition}

\newtheorem*{cor*}{Corollary}

\newcommand{\nc}{\newcommand}

\nc{\bA}{\mathbb A}
\nc{\bC}{\mathbb C}
\nc{\bc}{{\bf c}}
\nc{\bD}{\mathbb D}
\nc{\bd}{\mathbb d}
\nc{\bG}{\mathbb G}
\nc{\bi}{\bold i}
\nc{\bL}{\mathbb L}
\nc{\bN}{\mathbb N}
\nc{\bO}{\mathbb O}
\nc{\bP}{\mathbb P}
\nc{\bQ}{\mathbb Q}
\nc{\bR}{\mathbb R}
\nc{\bu}{\mathbb u}
\nc{\bv}{\bold v}
\nc{\bw}{\bold w}
\nc{\bW}{\mathbb W}
\nc{\bZ}{\mathbb Z}

\nc{\cA}{\mathcal A}
\nc{\cC}{\mathcal C}
\nc{\cD}{\mathcal D}
\nc{\cE}{\mathcal E}
\nc{\cF}{\mathcal F}
\nc{\cG}{\mathcal G}
\nc{\cH}{\mathcal H}
\nc{\cI}{\mathcal I}
\nc{\cK}{\mathcal K}
\nc{\cL}{\mathcal L}
\nc{\cM}{\mathcal M}
\nc{\cN}{\mathcal N}
\nc{\cO}{\mathcal O}
\nc{\cP}{\mathcal P}
\nc{\cR}{\mathcal R}
\nc{\cT}{\mathcal T}
\nc{\cU}{\mathcal U}
\nc{\cV}{\mathcal V}
\nc{\cW}{\mathcal W}
\nc{\cX}{\mathcal X}

\nc{\al}{\alpha}
\nc{\be}{\beta}
\nc{\la}{\lambda}
\nc{\La}{\Lambda}
\nc{\ve}{\varepsilon}
\nc{\om}{\omega}
\nc{\bPsi}{\boldsymbol{\Psi}}

\nc{\gl}{\mathfrak{gl}}
\nc{\fsl}{\mathfrak{sl}}
\nc{\g}{\mathfrak{g}}
\nc{\gh}{\widehat\g}
\nc{\h}{\mathfrak{h}}
\nc{\fb}{{\mathfrak b}}
\nc{\fg}{{\mathfrak g}}
\nc{\fgh}{{\widehat{\mathfrak g}}}
\nc{\fh}{{\mathfrak h}}
\nc{\fl}{\mathfrak{l}}
\nc{\fm}{{\mathfrak m}}
\nc{\fM}{{\mathfrak M}}
\nc{\fp}{{\mathfrak p}}
\nc{\ft}{\mathfrak{t}}

\nc{\fn}{{\mathfrak n}}
\nc{\fQ}{\mathfrak{Q}}

\nc{\Aut}{\mathrm{Aut}}
\nc{\ch}{{\mathop {\rm ch}}}
\nc{\tr}{{\mathop {\rm tr}\,}}
\nc{\im}{{\mathop {\rm im}}}
\nc{\id}{{\mathop {\rm id}}}
\nc{\ad}{{\mathop {\rm ad}}}
\nc{\gr}{\mathrm{gr}}
\nc{\ord}{\mathrm{ord}}
\nc{\red}{\mathrm{red}}
\nc{\End}{\operatorname{End}}
\nc{\Spec}{\operatorname{Spec}}
\nc{\Spf}{\operatorname{Spf}}
\nc{\Proj}{\operatorname{Proj}}
\nc{\Pic}{\operatorname{Pic}}
\nc{\Lie}{\operatorname{Lie}}
\nc{\Coh}{\mathrm{Coh}}
\nc{\coh}{\mathrm{coh}}
\nc{\qcoh}{\mathrm{Qcoh }}
\nc{\Gal}{\operatorname{Gal}}
\nc{\Hom}{\mathrm{Hom}}
\nc{\Rhom}{\mathrm{RHom}}
\nc{\cHom}{\mathcal{Hom}}
\nc{\Ann}{\mathrm{Ann}}
\nc{\Vect}{\mathrm{Vect}}
\nc{\wt}{\mathrm{wt}}
\nc{\hw}{\mathrm{hw}}
\nc{\rk}{\operatorname{rank}}
\nc{\Gr}{{\mathrm {Gr}}}
\nc{\Fl}{\mathrm{Fl}}
\nc{\spn}{\mathrm{span}}
\nc{\Rep}{\operatorname{Rep}}
\nc{\Irrep}{\mathrm{Irrep }}
\nc{\supp}{\operatorname{supp}}
\nc{\tp}{\mathrm{top}}
\nc{\codim}{\mathrm{codim}}
\nc{\IC}{\operatorname{IC}}
\nc{\modules}{\mathrm{-mod}}
\nc{\Perv}{\mathrm{Perv}}
\nc{\Forg}{\operatorname{Forg}}
\nc{\Maps}{\mathrm{Maps}}
\nc{\Frac}{\operatorname{Frac}}
\nc{\Stab}{\operatorname{Stab}}
\nc{\shhom}{\mathop{\mathcal{H}\! \mathit{om}}\nolimits}
\nc{\chom}{\mathop{\mathcal{H}\! \mathit{om}}\nolimits}
\nc{\cEnd}{\mathop{\mathcal{E}\! \mathit{nd}}\nolimits}
\nc{\mods}{\mathrm{-mod}}
\nc{\exo}{\mathrm{exo}}

\nc{\GfL}{{(G \times \bC^\times, \fl \oplus \bC D)}}
\nc{\Gfl}{{(G \times \bC^\times, \fl \oplus \bC D)}}

\nc{\eO}{\EuScript{O}}

\nc{\bra}{\langle}
\nc{\ket}{\rangle}
\nc{\pa}{\partial}
\nc{\ld}{\ldots}
\nc{\cd}{\cdots}
\nc{\hk}{\hookrightarrow}
\nc{\T}{\otimes}
\nc{\ov}{\overline}
\nc{\wh}{\widehat}
\nc{\wti}{\widetilde}
\nc{\svee}{{\!\scriptscriptstyle\vee}}
\nc{\ula}{{\underline{\la}}}
\nc{\umu}{{\underline{\mu}}}
\nc{\conv}{{\widetilde \times}}
\nc{\lach}{{\la^\svee}}
\nc{\alch}{{\al^\svee}}
\nc{\omch}{{\omega^\svee}}
\nc{\much}{{\mu^\svee}}
\nc{\md}{\text {--mod}}
\nc{\pt}{\mathrm{pt}}
\nc{\torus}{\bC^\times}

\theoremstyle{definition}
\newtheorem{theorem}{Theorem}

\newtheorem{corollary}[theorem]{Corollary}
\newtheorem{lemma}[theorem]{Lemma}

\newtheorem{proposition}[theorem]{Proposition}

\newtheorem{remark}[theorem]{Remark}

\mathcode`<="013C
\let\relprec\prec
\renewcommand{\prec}{{\relprec}}
\let\relll\ll
\renewcommand{\ll}{{\relll}}

\title{Irreducible objects in the Gaiotto category at roots of unity}
\author[Aleksandr Popkovich]{Aleksandr Popkovich}
\address{Department of Mathematics, University of Toronto}
\email{alexander.popkovich@mail.utoronto.ca}

\begin{document}

\begin{abstract}
 A theorem of R. Travkin and R. Yang, initially conjectured by D. Gaiotto, states that for a generic (not a root of unity) $q$ the category of  $q$-twisted D-modules on the affine Grassmannian $Gr_{GL_N}$ which are equivariant with respect to a certain subgroup (defined by a choice of $0 \le M <N$) of $GL_N$ is equivalent to the category of representations of the quantum supergroup $U_q(\mathfrak{gl}(M|N))$. We aim to see whether this equivalence should hold when $q$ is a root of unity. We begin by asking if there is a natural bijection between the sets of irreducible objects. In this note we make an observation that suggests this should be the case: we show that there is a natural bijection between irreducible objects in the Gaiotto category and in the category of representations of a supergroup $GL(M|N)$ in positive characteristic. The proof is based on the version of the Serganova's algorithm formulated by J. Brundan and J. Kujawa in \url{	arXiv:math/0210108}.
\end{abstract}

\maketitle

\section{Introduction}

Let $\mathcal{K} :=\mathbb{C}((t))$ be the field of Laurent series in the variable $t$, $\mathcal{O} = \mathbb{C}[[t]] \subset \mathcal{K}$ be its ring of integers and $\mathrm{Gr}_N = GL_N(\mathcal{K})/GL_N(\mathcal{O})$ be the Affine Grassmannian of $GL_N$.

Fix nonnegative integers $M<N$. One can define a certain unipotent subgroup $U_{M,N}^{-}(\mathcal{K}) \subset GL_N(\mathcal{K})$ and its character $\chi: U_{M,N}^{-}(\mathcal{K})\to \mathbb{C}^\times$ which are normalized by $GL_M(\mathcal{K}) \subset GL_N(\mathcal{K})$ (see \cite{ty} for the definitions and relevant discussion).

Travkin and Yang (\cite{ty}) proved the following:

\begin{theorem}
If $q$ is not a root of unity, there is a t-exact monoidal equivalence
\begin{equation}
\label{TY_theorem}
\mathrm{SD}_q^{\mathrm{GL}_M(\mathcal{O}) \ltimes U_{M,N}^{-}(\mathcal{K}), \chi}(\mathrm{Gr}_N) \simeq \mathrm{Rep}^q(\mathrm{GL}(M|N))),
\end{equation}
where $\mathrm{SD}_q$ denotes the derived category of $q$-twisted $D$-modules with coefficients in a super vector space $\mathrm{SVect}$, and $\mathrm{Rep}^q(\mathrm{GL}(M|N))$ is the category of finite-dimensional representations of the quantum supergroup $U_q(\mathfrak{gl}(m|n))$.
\end{theorem}

 We aim to see whether this equivalence has a chance to hold when $q$ is a root of unity. We begin by asking if in that case there is a natural bijection between the sets of irreducible objects of the categories in question.

In this note we make an observation that suggests this should be the case.

The irreducible objects of the left-hand side of \ref{TY_theorem} are parametrized by (extended from) a subset of orbits of $\mathrm{GL}_M(\mathcal{O}) \ltimes U_{M,N}^{-}(\mathcal{K})$ on $\mathrm{Gr}_N$ - these orbits are called relevant. All orbits (relevant or not) of this group on $\mathrm{Gr}_N$ are parametrized by dominant weights of $GL(M|N)$, that is by highest weights of irreducible modules of $GL(M|N)$ with respect to the standard Borel subalgroup. Now, unlike in the usual theory, in case of supergroups not all of the Borel subgroups are conjugate to each other - in particular there is so called ``mixed" Borel subgroup, which in some sense is furthest away from the standard one. For generic $q$ it was observed in \cite{ty} that the relevant orbits are parametrized by highest weights of irreducible representations of $U_q(\mathfrak{gl}(m|n))$ with respect to the mixed Borel subalgebra - these two sets are the same as subsets of the lattice of integral weights of the supergroup. We aim to see that the same happens for general $q$. 

However, we cannot literally check this: one of the difficulties is that while one build a ``quantum supergroup $U_q(\mathfrak{gl}(m|n))$" starting with any Borel subalgebra, we don't even know if the quantum supergroups defined from different Borels are the same.

To work around this, we rely on a well-known principle, saying that a category of representations of a quantum group at roots of unity of a prime order $p$ should be closely related to the category of representations of a (non-quantum) group over a field $\mathbf{k}$ of positive characteristic $p$. In our case, we consider the category of representations in characteristic $p$ of the supergroup $GL(m|n)$, and we denote this category $\mathrm{Rep}_p(\mathrm{GL}(M|N))$. 

One modification we need to make compared with \cite{ty} is to switch to a slightly different unipotent subgroup of $GL_N(\mathcal{K})$: the choice of the unipotent subgroup $U_{M,N}^-(\mathcal{K})$ is not the only possible one, and in fact is different from the choice that Travkin and Yang use in their treatment of the untwisted case in \cite{ty_untwisted}, where they use a subgroup they denote $U_{M,N}(\mathcal{K})$ which consists of matrices that are transposes of those in $U_{M,N}^-(\mathcal{K})$. However, either choice leads to equivalent categories (see Section 5 of \cite{ty_untwisted}), and the effect of a different choice on the combinatorics is reversal of signs in inequalities that describe relevant weights and orbits. Using the group $U_{M,N}(\mathcal{K})$ makes our result just a little cleaner.

The result that we prove in this note is the following.
\begin{theorem}
    Let $q^p=1$ for prime $p$. Then the set of orbits that support irreducible objects of $\mathrm{SD}_q^{\mathrm{GL}_M(\mathcal{O}) \ltimes U_{M,N}(\mathcal{K}), \chi}(\mathrm{Gr}_N)$ is equal (as a subset of the lattice of integral weights of supergroup $GL(m|n)$) to the set of highest weights of irreducible objects of $\mathrm{Rep}_p(\mathrm{GL}(M|N))$ with respect to the ``mixed" Borel subalgebra.
\end{theorem}

The proof is based on the work of Brundan and Kujawa \cite{bk} who formulated a version of Serganova's algorithm which allows one to determine highest weights of objects of $\mathrm{Rep}_p(\mathrm{GL}(M|N))$ with respect to different Borel subalgebras.

\subsection*{Acknowledgements}
We are grateful to Alexander Braverman for posing the question that led to this note and for many helpful discussions, and to Pavel Etingof for pointing out the work \cite{bk} to us.

\section{Irreducible objects in the $\mathrm{Rep}_p(\mathrm{GL}(M|N))$}
 
In this sectioin we study the set ofe irreducible objects of the category $\mathrm{Rep}_p(\mathrm{GL}(M|N))$, which models the category in the right-hand side of \ref{TY_theorem} for $q^p=1$.
\subsection{Borel subgroups and highest weights}

Given a Borel subgroup $B\subset G = GL(M|N)$ one can define a notion of a dominant weight and of the highest weight of a representation of $G$ the same way it is done in the classical theory of representations of reductive groups. Similarly to the classical theory, it then turns out that finite-dimensional irreducible representations of $G$ are in bijection with the set of dominant integral weights. 

However, in the context of supergroups the resulting picture depends on the choice of the Borel subgroup $B \subset G$ - this is because the Borel subgroups of the supergroup $G$ are not all conjugate to each other. 

To be more precise, let us choose a homogeneous basis $v_1,...,v_M,v_{M+1},...v_{M+N}$ of $\mathbb{C}^{(M,N)}$, with $v_i$ even for $1\le i \le M$ and odd for $M+1 \le i \le M+N$. This choice fixes a Cartan subgroup in $GL(M|N)$ and a Cartan subalgebra in $\mathfrak{gl}(M|N)$. We will consider this choice fixed from now on.

For any permutation $\omega \in S_{m+n}$ denote by $Fl_\omega$ the flag
\begin{equation}
    Fl_\omega := span_\mathbb{C}\langle v_{\omega(1)}\rangle \subset span_\mathbb{C}\langle v_{\omega(1)},v_{\omega(2)}\rangle \subset ...\subset span_\mathbb{C}\langle v_{\omega(1)},...,v_{\omega(M+N)}\rangle.
\end{equation}

Then any Borel subgroup $B \subset G$ is the stabilizer of $FL_\omega$ for some $\omega \in S_{M+N}$ and two such Borels $B_{\omega_1}$ and $B_{\omega_2}$ are conjugate to each other if and only if the parity of $v_{\omega_1(i)}$ and $v_{\omega_2(i)}$ is the same for all $i \in \{1,...,M+N\}$.

Below we will be interested in two particular Borel subgroups of $G$: denote by $B_{s}:=Stab(Fl_{id})$  (``s" stands for ``standard") the stabilizer of the standard flag 
\begin{equation}
    Fl_{id}:= span_\mathbb{C}\langle v_1\rangle \subset span_\mathbb{C}\langle v_{1},v_{2}\rangle \subset ...\subset span_\mathbb{C}\langle v_{1},...,v_{M+N}\rangle
\end{equation}
and by $B_{m}$ (``m" stands for ``mixed") the stabilizer of the flag with maximal possible number of ``switches" between even and odd vectors:
\begin{multline*}
    Fl_{m}:= span_\mathbb{C}\langle v_{M+1}\rangle \subset span_\mathbb{C}\langle v_{M+1}, v_{1}\rangle \subset span_\mathbb{C}\langle v_{M+1},v_{1},v_{M+2}\rangle\subset \\
    \subset span_\mathbb{C}\langle v_{M+1},v_{1},v_{M+2}, v_{2}\rangle\subset...\subset span_\mathbb{C}\langle v_{M+1},v_{1},v_{M+2}, v_{2},...,v_{M+N}\rangle.
\end{multline*}

Identify the weight lattice of the Cartan subalgebra of $GL(M|N)$ with $\mathbb{Z}^M \times \mathbb{Z}^N$ using as generators the diagonal weights $\varepsilon_1, ..., \varepsilon_M,\varepsilon_{M+1},..., \varepsilon_{M+N}$, with $\varepsilon_i$ even for $i=1,..,M$ and odd otherwise.

Then the highest weights of finite-dimensional representations with respect to the standard Borel subgroup $B_s$ are identified with the tuples $(\lambda,\theta)\in\mathbb{Z}^M\times\mathbb{Z}^{N}$ with $\lambda_1 \ge \lambda_2 \ge ... \ge \lambda_M\ \text{ and }\theta_1\ge \theta_2 ... \ge \theta_{M+1}$.

We will also need the following description of the positive roots of the borel subgroups: under the identification of the weight lattice with $\mathbb{Z}^M \times \mathbb{Z}^N$ desribed above, the positive roots of some Borel subgroup $B_\omega$ are
\begin{equation}
    \Phi^+_\omega = \{ \varepsilon_{\omega(i)} - \varepsilon_{\omega(j)}| 1 \le i < j \le M+N\}.
\end{equation}

For more details on all of the above see \cite{ty}, \cite{s}.

Now we aim to describe the highest weight of finite-dimensional representations with respect to the mixed Borel group $B_{m}$. We will later see that this set also parametrizes the irreducible objects in the Gaiotto category $\mathrm{SD}_q^{\mathrm{GL}_M(\mathcal{O}) \times U_{M,N}^{-}(\mathcal{K}), \chi}(\mathrm{Gr}_N)$. 

To provide the description of the highest weights of $B_{m}$ we will use a version of Serganova's algorithm developed in \cite{bk}. This algorithm, described in the following (rather technical) subsection, allows one to see how the set of the highest weight changes when passing from one Borel subgroup to another, therefore allowing us to go from the above description of the highest weights of the standard Borel to those of the mixed one.

\subsection{Serganova's algorithm}

Our goal now is to discribe the set of highest weights of irreducible representations from the category $Rep_p(GL(M|N))$ with respect to the ``mixed" Borel subgroup. For this we utilize a version of Serganova's algorithm formulated in \cite{bk} (see Theorem 4.3. in loc. cit.). The algorithm allows us to describe the set of highest weights for ``mixed" Borel starting with the set of highest weights of the standard one. 

When discussing the algorithm we assume $N=M+1$ - this slightly simplifies notation and does not lead to loss of generality: if $N>M+1$ then $\theta_{M+2},...,\theta_{N}$ are ``dummy" entries, unaffected by the algorithm. Also, below we formulate the algorithm completely in terms of operations with tuples of integers and only for the case of passing from the stndard to the ``mixed" Borel - for more general discussion we refer to \cite{bk}.

Now consider 
\begin{equation}
    A := \{(\lambda,\theta)\in\mathbb{Z}^M\times\mathbb{Z}^{M+1} \text{ such that } \lambda_1 \ge \lambda_2 \ge ... \ge \lambda_M\ \text{ and }\theta_1\ge \theta_2 ... \ge \theta_{M+1}\}
\end{equation} 

This is the set of dominant integral weights of $GL(M|N)$, i.e. the highest weights with respect to the standard Borel subgroup (see Section 4 of \cite{bk}).

There is a map $S:A \to \mathbb{Z}^M\times\mathbb{Z}^{M+1}$ given by applying to the elements of $A$ the following algorithm.

The algorithm (but not its result) will depend on a choice of a linear order on the set $\{\alpha_i\} := \Phi^+_{st}\backslash \Phi^+_m$ which refines the "standard" partial order given by $\alpha \preccurlyeq_{st} \beta \iff (\alpha-\beta \in \Phi_{st}^+)$.
    
There are two linear orders that we will use. Under the identification $\Phi^+_{st}\backslash \Phi^+_m =\{(i,j)|1 \le i \le M,1\le j< M+1, j \le i\}$ the ``standard"  partial order is given by
\begin{equation}
    (i_1,j_1) \preccurlyeq (i_2,j_2) \iff i_1\ge i_2 \text{ and } j_1 \le j_2
\end{equation}
and there are two natural linear refining of it:
\begin{enumerate}
    \item $(M,1)<_1(M-1,1)<_1(M-2,1)<_1...<_1(1,1)<_1(M,2)<_1...<_1(M,M)$
    \item $(M,1)<_2(M,2)<_2...<_2(M,M)<_2(M-1,1)<_2...<_2(1,1)$
\end{enumerate}
    
\begin{definition}
    (Serganova's algorithm for $p > 0$). 
    
    The two linear orderings described above lead to two slightly different versions of the algorithms
    \begin{enumerate}
        \item Version 1: we go through the set $\Phi^+_{st}\backslash \Phi^+_m$ using the linear order $<_1$ described above, or
        \item Version 2: we use the linear order $<_2$.
    \end{enumerate}    
    (the result does not depend on the version used).
    
    In both versions, on the $k$-th step of the algorithm transform one of the following two actions is taken: if $(i_k,j_k)$ is the $k$-th element of $\Phi^+_{st}\backslash \Phi^+_m$ with respect to the chosen order then 
    \begin{enumerate}
        \item (Action 1) If $\lambda^{(k)}_{i_k}+\theta^{(k)}_{j_k} \equiv 0 \text{ mod }p$, then we do nothing: $\lambda^{(k+1)} = \lambda^{(k)}$ and $\theta^{(k+1)}=\theta^{(k)}$
        \item (Action 2) Otherwise, $\lambda^{(k+1)}_{i_k} = \lambda^{(k)}_{i_k}-1$, $\theta^{(k+1)}_{j_k}=\theta^{(k)}_{j_k}+1$ and for all $i\neq i_k$ and $j \neq j_k$ we have $\lambda^{(k+1)}_{i} = \lambda^{(k)}_{i_k}$ and $\theta^{(k+1)}_{j}=\theta^{(k)}_{j}$.
    \end{enumerate}.
\end{definition}

We are able to describe the image of the set $A$ under the Serganova's algorithm.

\begin{lemma}
Denote by $M \subset A$ the subset consisting of pairs $(\lambda, \theta)$ that satisfy the following  conditions:
\begin{enumerate}
    \item $\lambda_1 \ge \lambda_2 \ge ... \ge \lambda_M$ and $\theta_1\ge \theta_2 ... \ge \theta_{M+1}$
    \item if $\lambda_{i-1} =\lambda_i$ then $\lambda_i+\theta_i=0 \text{ mod }p$, and
    \item if $\theta_{i} =\theta_{i+1}$ then $\lambda_i+\theta_i=0 \text{ mod }p$.
\end{enumerate}

    Then we have an equality of sets $S(A)=M \subset \mathbb{Z}^M\times\mathbb{Z}^{M+1}$.
\end{lemma}

\begin{proof}
    First, let us show that $S(A)\subset M$. We will use the fact that the linear ordering of $\Phi^+_{st}\backslash \Phi^+_m$ can be chosen in any way which is consistent with the partial order on this set without affecting the result of the algorithm.
    \begin{itemize}
        \item To show that any element $(\tilde\lambda,\tilde\theta)=S\left( (\lambda,\theta)\right)$ in the image of $S$ has $\tilde\lambda_{i} \ge \tilde\lambda_{i+1}$ for all $i$  and the condition (i) of the Lemma is satisfied, we use Version 1 of the algorithm.

        First we show that $\tilde\lambda_{i} \ge \tilde\lambda_{i+1}$ for all $i$, that is $\tilde\lambda$ is a non-increasing sequence of integers. In fact, a stronger statement is true (and we will need it swiftly): the property of being non-increasing is preserved on every step of the algorithm. Indeed, for the sake of contradiction suppose that on $k$-th step of the algorithm this property is broken for the first time, that is $\lambda^{(k)}$ was non-increasing but $\lambda^{(k+1)}$ is not. Since only the $i_k$-th entry of $\lambda$ could be affected on this step, it must have been changed (decreased) and we must have $\lambda^{(k+1)}_{i_k}<\lambda^{(k+1)}_{i_k+1}$. Since every time Action 2 is taken the entry is only decreased by one, and on the previous iteration $\lambda^{(k)}$ was non-increasing, we must have had $\lambda^{(k)}_{i_k}<\lambda^{(k)}_{i_k+1}$, and on a previous $(k-1)$-th step with $i_{k-1}=i_k+1$ (since we are using the dirst ordering) Action 1 must have been taken. But then it also should have been taken on the $k$-th step. A contradiction.

         Now to prove that condidtion (i) is always satisfied for $(\tilde\lambda,\tilde\theta)$, suppose that we have $\tilde\lambda_{i-1} =\tilde\lambda_{i}$ for some $i$. Let $k$ be the last step on which $i_k=i$ so that $\lambda^{(k+1)}_{i}=\tilde\lambda_{i}$. Since this is the last step on which $i_k=i$ and we are using the ordering $<_1$ the following must be true: we must have $j_k=i_k$, and all of the remaining ``m"-th steps have $i_m>i_k$ and $j_m>j_k$. It follows that $\lambda_{i_k-1}^{(k+1)}= \tilde\lambda_{i_k-1}$, $\lambda_{i_k}^{(k+1)}= \tilde\lambda_{i_k}$ and $\theta_{i_k}^{(k+1)}= \tilde\theta_{i_k}$. So due to our assumption we must have $\lambda_{i-1}^{(k+1)} =\lambda_{i}^{(k+1)}$. This, together with the fact (proven above) that $\lambda^{(k)}$ was non-increasing implies that on $k$-th step Action 1 was taken, so $\lambda_{i_k}^{(k+1)}= \theta_{i_k}^{(k+1)}$ and hence $\tilde\lambda_{i_k}^{(k+1)}= \tilde\theta_{i_k}^{(k+1)}$.

        \item Analogous reasoning using Version 2 of the algorithm shows that any element $(\tilde\lambda,\tilde\theta)=S\left( (\lambda,\theta)\right)$ in the image of $S$ has $\tilde\theta_{i} \ge \tilde\theta_{i+1}$ for all $i$  and the condition (ii) of the Lemma is satisfied.

        \item Finally, let us show that any pair $(\tilde\lambda,\tilde\theta) \in M$ can be obtained as $S((\lambda,\theta))$ for some $(\lambda,\theta) \in A$. It turns out one can describe an inverse map $S^{-1}:M\to A$.

        For this note that whichever of the two Actions of the algorithm is take on step $k$, we have $\lambda^{(k)}_{i_k}+\theta^{(k)}_{j_k} \equiv \lambda^{(k+1)}_{i_k}+\theta^{(k+1)}_{j_k} \text{ mod }p$; so at every moment we "remember" what Action we took on the previous step. Therefore, we can "invert" the algorithm by going through the set $\Phi^+_{st}\backslash \Phi^+_m$ in the order inverse to either $<_1$ or $<_2$ and performing the transformations inverse to the ones we did earlier.

        This defines a map $S^{-1}:M \to \mathbb{Z}^M\times\mathbb{Z}^{M+1}$; we need to show that $S^{-1}(M)\subset A$, i.e. that for any element $(\tilde\lambda,\tilde\theta)\in M$ and $(\lambda,\theta):=S^{-1}((\tilde\lambda,\tilde\theta))$ both $\lambda$ and $\theta$ are non-decreasing.

        Let us prove that $\lambda$ is non-decreasing.  For this we will use the inverse to the Version 1 of the Serganova algorithm and the condition (i) of the Lemma. Since $\tilde\lambda$ is already non-decreasing, we just need to show that this property is preserved on each step of the procedure. 

        Let $k$-th step of the reversed algorithm be the first one on which the property of being non-increasing is broken. Call the input of this step $(\lambda^{(in)},\theta^{(in)})$ and the output $(\lambda^{(out)},\theta^{(out)})$. The property being broken means that there are indeces $i$ and $j$ (indeces of entries being considered on this step) such that  $\lambda^{(in)}_{i-1}=\lambda^{(in)}_{i}$, but $\lambda^{(out)}_i= \lambda^{(out)}_{i-1}+1$ and $\theta^{(out)}_j = \theta^{(in)}_j-1$, that is on this step the inverse of Action 2 was performed, increasing the $i$-th value of $\lambda$ and decreasing the $j$-th values of $\theta$. Now consider two cases: either $i=j$ or $i>j$. 

        If $i>j$ then the above could not happen: for the inverse of Action 2 to be taken we must have $\lambda^{(in)} + \theta^{(in)} \ne 0\text{ mod }p$, but to have $\lambda^{(in)}_{i-1}=\lambda^{(in)}_{i}$ Action 1 must have been taken on the previous step  - one easily sees a contradiction.

        If $i=j$  then again the inverse to Action 2 could not have been performed on this step: due to how reverse to order $<_1$ works, we must have $\lambda^{(in)}_{i-1}=\tilde\lambda_{i-1} $, $\lambda^{(in)}_{i}= \tilde\lambda_{i}$ and $\theta^{(in)}_{i}= \tilde\theta_{i}$, and we know that any element $(\tilde\lambda,\tilde\theta)\in S(A)$ satisfies condition (i) of the Lemma, so from $\lambda^{(in)}_{i-1}=\lambda^{(in)}_{i}$ follows $\lambda^{(in)} + \theta^{(in)} = 0\text{ mod }p$ hence the inverse to Action 2 was not taken.

        Proof that $\theta$ is non-decreasing is analogues, using the inverse to Version 2 of the algorithm and condition (ii) of the lemma.

    \end{itemize}
\end{proof}

\begin{corollary}
    \label{corollary_weights}
    The set of highest weights of irreducible representations from the category $Rep_p(GL(M|N))$ with respect to the ``mixed" Borel subgroup consists of tuples $(\lambda,\theta) \in \mathbb{Z}^M\times \mathbb{Z}^N$ that satisfy the following conditions.
\begin{equation}\label{condition 3.4}
  \begin{split}
    \textnormal{A}:    \textnormal{A}:    \lambda_1\geq\lambda_2\geq\cdots\geq \lambda_M,\\
        \theta_{1}\geq \theta_2\geq\cdots\geq\theta_{M+1}\geq\theta_{M+2}\geq\cdots \geq \theta_{N}\\
        \textnormal{and B: if }  \theta_i=\theta_{i+1}, \textnormal{then } \theta_i+\lambda_i=0 \text{ mod }p, \textnormal{for } i=1,2,..., M;\\
        \textnormal{if } \lambda_{i-1}=\lambda_i, \textnormal{then } \theta_i+\lambda_i=0 \text{ mod }p, \textnormal{for } i=2,..., M.
    \end{split}
\end{equation}
\end{corollary}

\section{Relevant orbits}

Now let us discuss the irreducible objects in the Gaiotto category $\mathrm{SD}_q^{\mathrm{GL}_M(\mathcal{O}) \times U_{M,N}^{-}(\mathcal{K}), \chi}(\mathrm{Gr}_N)$, i.e. in the left-hand side of the desired equivalence \ref{TY_theorem}. For standard reasons, they correspond to a subset of orbits of the action of $\mathrm{GL}_M(\mathcal{O}) \times U_{M,N}^{-}(\mathcal{K})$ on $\mathrm{Gr}_N$. These are the orbits that support the (twisted) $D$-modules that satisfy the necessary equivariance conditions; we call such orbits relevant.

\subsection{Results of Travkin and Yang}

A detailed study of relevant orbits for the case when $q$ is not a root of unity is can be found in \cite{ty}, and similar analyses were performed in \cite{bft1}, \cite{bft2} and \cite{ty_untwisted}. 

In this subsection we state the classification from \cite{ty} with one modification (accounting for the condition $q^p=1$) and indicate where this modification is coming from. 

First, one sees that all $GL_M(\mathbb{O})\ltimes U_{M,N}^-(\mathcal{K})$-orbits are enumerated by certain weights:

\begin{lemma}\label{H orbit} (Lemma 3.3.3 in \cite{ty})
Any $GL_M(\mathbb{O})\ltimes U_{M,N}^-(\mathcal{K})$-orbit in $\Gr_N$ has a unique representative of the form $\mathbb{L}_{(\lambda,(\theta, \theta'))}$, such that $(\lambda,\theta)$ is a dominant weight of length $(M,M+1)$, and $\theta'$ is a sequence of length $N-M-1$. Here,
\begin{equation}
    \mathbb{L}_{(\lambda,(\theta, \theta'))}=\begin{pmatrix}
    t^{-\lambda_1-\theta_1}       &   &  &  &&&\\
    &\ddots&& &&&\\
          & &   t^{-\lambda_M-\theta_M} &&&& \\
         t^{-\theta_1}  &\dots&  t^{-\theta_M} & t^{-\theta_{M+1}}&&&\\
         &&&&t^{-\theta'_1}&&\\
         &&&&&\ddots&\\
         &&&&&&t^{-\theta'_{N-M-1}}
\end{pmatrix}\in GL_{N}(\mathcal{K}).
\end{equation}
\end{lemma}

Then one can work out which of them are relevant (thus also classifying the irreducible objects in the Gaiotto category). 

\begin{proposition}\label{Proposition_ty}(Proposition 3.3.4 in \cite{ty})
    The  $GL_M(\mathbb{O})\ltimes U_{M,N}^-(\mathcal{K})$-orbit of $\mathbb{L}_{(\lambda,(\theta, \theta'))}$ is relevant if and only if $(\lambda,(\theta, \theta')) \in M$, that is if and only if
\begin{equation}\label{condition 3.4}
  \begin{split}
    \textnormal{A}:    \lambda_1\leq\lambda_2\leq\cdots\leq \lambda_M,\\
        \theta_{1}\leq \theta_2\leq\cdots\leq\theta_{M+1}\leq\theta'_1\leq\cdots \leq \theta'_{N-M-1},\\
        \textnormal{and B: if }  \theta_i=\theta_{i+1}, \textnormal{then } \theta_i+\lambda_i=0 , \textnormal{for } i=1,2,..., M;\\
        \textnormal{if } \lambda_{i-1}=\lambda_i, \textnormal{then } \theta_i+\lambda_i=0 , \textnormal{for } i=2,..., M.
    \end{split}
\end{equation}
\end{proposition}
For the case of $q^p=1$ we claim the following.
\begin{proposition}
\label{Proposition_ty_root} Suppose $q^p=1$. Then 
    The  $GL_M(\mathbb{O})\ltimes U_{M,N}^-(\mathcal{K})$-orbit of $\mathbb{L}_{(\lambda,(\theta, \theta'))}$ is relevant if and only if $(\lambda,(\theta, \theta')) \in M$, that is if and only if
\begin{equation}\label{condition 3.4}
  \begin{split}
    \textnormal{A}:    \textnormal{A}:    \lambda_1\leq\lambda_2\leq\cdots\leq \lambda_M,\\
        \theta_{1}\leq \theta_2\leq\cdots\leq\theta_{M+1}\leq\theta'_1\leq\cdots \leq \theta'_{N-M-1}\\
        \textnormal{and B: if }  \theta_i=\theta_{i+1}, \textnormal{then } \theta_i+\lambda_i=0 \text{ mod }p, \textnormal{for } i=1,2,..., M;\\
        \textnormal{if } \lambda_{i-1}=\lambda_i, \textnormal{then } \theta_i+\lambda_i=0 \text{ mod }p, \textnormal{for } i=2,..., M.
    \end{split}
\end{equation}
\end{proposition}

\begin{proof}
    The proof is the same as the proof of Proposition 3.3.4 in \cite{ty} with only a minor modification. 

    Just as in loc. cit. we need to check that the condition stated above is equivalent to the condition that $Stab_{GL_M(\mathbb{O})\ltimes U^-_{M,N}(\mathcal{K})}(\mathbb{L}_{(\lambda,(\theta,\theta'))})\subset Ker\ \chi$ and that this stabilizer acts trivially on the fiber of $\overset{\circ}{\mathcal{P}}_{\det, N}$ over $\mathbb{L}_{(\lambda,(\theta,\theta'))}\in Gr_N$. The proof that $Stab_{GL_M(\mathbb{O})\ltimes U^-_{M,N}(\mathcal{K})}(\mathbb{L}_{(\lambda,(\theta,\theta'))})\subset Ker\ \chi$ iff $\theta_{M+1}\leq \theta'_1\leq\cdots\leq \theta'_{N-M-1}$ is the same as in loc cit. 

    The minor change should be introduced to the proof that the stabilizer group $Stab_{GL_M(\mathcal{O})}(\mathbb{L}_{(\lambda,\theta)})$ acts trivially on the fiber over $\mathbb{L}_{(\lambda,\theta)}$ if and only if $(\lambda,\theta)$ satisfies the condition B (since the condition B in loc.cit. does not feature any comparison mod $p$). Namely, \cite{ty} refers for the proof to  Proposition 4.1.1 of \cite{bft1}, which in turn is based on the computation of stabilizers made in \cite{bft2}, Lemma 2.3.3. In this computation one encounters the action of some $GL_n$ (a factor of $Stab_{GL_M(\mathcal{O})}(\mathbb{L}_{(\lambda,\theta)})$) acting on  the fiber over $\mathbb{L}_{(\lambda,\theta)}$ by $det^a$ where $a=\theta_i+\lambda_i$ for $i$ such that either $\theta_i=\theta_{i+1}$ or $\lambda_{i-1}=\lambda_i$. 
    
    When $q$ is not a root of unity, the condition that this action trivializes is that $a=0$, hence the form of condition B of Proposition 4.1.1 of \cite{ty}. In our case when $q^p=1$, this action trivializes whenever $a=0 \text{ mod }p$, which is the reason for the modification of condition B above.
\end{proof}

\subsection{Swithing the group}
It is time to make an important remark.

\begin{remark}
    We want to note at this point that the choice of the unipotent subgroup $U_{M,N}^-(\mathcal{K})$ of $GL_N(\mathcal{K})$ made in \cite{ty} is not the only possible one, and in fact is different from the choice in \cite{ty_untwisted}, where the authors define and use the group $U_{M,N}(\mathcal{K})$ which is ``transposed" $U_{M,N}^-(\mathcal{K})$ (i.e. $U_{M,N}(\mathcal{K})$ is a subgroup of upper-triangular matrices, while $U_{M,N}^-(\mathcal{K})$ is in lower-triangular). Both choices in fact lead to equivalent categories of equivariant sheaves - see discussion in section 5 of \cite{ty_untwisted}. However, the descriptions of all orbits and of relevent orbits don't change drastically: when one switches from $U_{M,N}^-(\mathcal{K})$ to $U_{M,N}(\mathcal{K})$, all of the inequalities should be reversed (which is proved mutatis mutandis compared to the proof of Lemma 3.3.3 and Proposition 3.3.4 of \cite{ty}).
\end{remark}

\begin{corollary}
\label{corollary_orbits} Suppose $q^p=1$. The set of all  $GL_M(\mathbb{O})\ltimes U_{M,N}(\mathcal{K})$-orbits is in bijection with integral weights of $GL(M|N)$, and the set of relevant orbits consists of those tuples $(\lambda,(\theta, \theta')) \in \mathbb{Z}^M\times \mathbb{Z}^N$ that satisfy the following conditions.
\begin{equation}\label{condition 3.4}
  \begin{split}
    \textnormal{A}:    \textnormal{A}:    \lambda_1\geq\lambda_2\geq\cdots\geq \lambda_M,\\
        \theta_{1}\geq \theta_2\geq\cdots\geq\theta_{M+1}\geq\theta'_1\geq\cdots \geq \theta'_{N-M-1}\\
        \textnormal{and B: if }  \theta_i=\theta_{i+1}, \textnormal{then } \theta_i+\lambda_i=0 \text{ mod }p, \textnormal{for } i=1,2,..., M;\\
        \textnormal{if } \lambda_{i-1}=\lambda_i, \textnormal{then } \theta_i+\lambda_i=0 \text{ mod }p, \textnormal{for } i=2,..., M.
    \end{split}
\end{equation}
\end{corollary}

We can now combine Corollary \ref{corollary_weights} and Corollary \ref{corollary_weights} and fomulate the main statement of this note.

\begin{theorem}
    Let $q^p=1$ for prime $p$. Then the set of orbits that support irreducible objects of $\mathrm{SD}_q^{\mathrm{GL}_M(\mathcal{O}) \ltimes U_{M,N}(\mathcal{K}), \chi}(\mathrm{Gr}_N)$ is equal (as a subset of the lattice of integral weights of supergroup $GL(m|n)$) to the set of highest weights of irreducible objects of $\mathrm{Rep}_p(\mathrm{GL}(M|N))$ with respect to the ``mixed" Borel subalgebra.
\end{theorem}



\begin{thebibliography}{99}

\bibitem[BK]{bk} J. Brundan, J. Kujawa  - \textit{A new proof of the Mullineux conjecture}, \url{arxiv.org/abs/math/0210108}
\bibitem[BFT1]{bft1}
A. Braverman, M. Finkelberg, and R. Travkin - \textit{Gaiotto conjecture for $Rep_q(GL(N-1,N))$}, \url{arxiv.org/abs/2107.02653v1}

\bibitem[BFT2]{bft2}
A. Braverman, M. Finkelberg, R. Travkin - \textit{Orthosymplectic Satake equivalence}, \url{https://arxiv.org/abs/1912.01930}

\bibitem[G]{G}D.Gaitsgory - \textit{Twisted Whittaker model and factorizable sheaves}, Selecta Math. (N.S.) 13
(2008), no. 4, 617–659; arXiv:0705.4571.

\bibitem[S]{s} V. Serganova - \textit{Representations of Lie Superalgebras}. Perspectives in Lie theory, Springer INdAM
Ser., 19, Springer, Cham (2017), 125–177

\bibitem[TY1]{ty_untwisted}
Roman Travkin, Ruotao Yang - \textit{Untwisted Gaiotto equivalence}, 	\url{arxiv.org/abs/2201.10462}

\bibitem[TY2]{ty}
Roman Travkin, Ruotao Yang - \textit{Twisted Gaiotto equivalence for $GL(M|N)$}, 	\url{arxiv.org/abs/2306.09556}




\end{thebibliography}
\end{document}